\documentclass[12pt]{article}
\usepackage{amssymb,amsmath,amsthm}
\usepackage{mathptmx}
\usepackage[pagebackref=true]{hyperref}
\usepackage{cite}
\usepackage{geometry}
\usepackage{xcolor}
\usepackage{comment}

\renewcommand*\backref[1]{\ifx#1\relax \else \mbox{\textcolor{gray}{$\uparrow$ #1}} \fi}

\hypersetup{
	colorlinks,%
	citecolor=blue,%
	filecolor=black,%
	urlcolor=black
}

\newtheorem{theorem}{Theorem}[section]
\newtheorem{lemma}[theorem]{Lemma}

\newtheorem{corollary}[theorem]{Corollary}
\numberwithin{equation}{section}

\newcommand\blfootnote[1]{%
  \begingroup
  \renewcommand\thefootnote{}\footnote{#1}%
  \addtocounter{footnote}{-1}%
  \endgroup
}

\newcommand{\N}{{\mathbb N}} 
\newcommand{\R}{{\mathbb R}}
\newcommand{\Rn}{{\mathbb R}^n}
\newcommand{\RN}{{\mathbb R}^{n+1}}
\newcommand{\s}{\mathbb{S}}
\newcommand{\sn}{\mathbb{S}^{n-1}}
\newcommand{\sN}{\mathbb{S}^{n}}
\newcommand{\SO}{{\rm SO}}

\newcommand{\KN}{{\mathcal{K}}^{n+1}}

\newcommand{\hm}{\mathcal H}

\newcommand{\fconvOmega}{{\mathrm{Conv}(\Omega)}} 
\newcommand{\Conv}{{\mathrm{Conv}}} 
\newcommand{\fconvcd}{{\mathrm{Conv}_{\mathrm{cd}}(\R^n)}} 
\newcommand{\proj}{\operatorname{proj}}
\newcommand{\gnom}{\operatorname{gno}} 
\newcommand{\vol}{\operatorname{vol}} 
\newcommand{\ospan}{\operatorname{span}} 

\newcommand{\MA}{\mathrm{MA}} 

\newcommand{\Hess}{{\operatorname{D}}^2}

\renewcommand{\d}{\,\mathrm{d}}
\newcommand{\bd}{\operatorname{bd}} 
\newcommand{\dom}{\operatorname{dom}} 

\title{Explicit solutions to Christoffel--Minkowski problems and Hessian equations under rotational symmetries}
\author{Fabian Mussnig and Jacopo Ulivelli}

\date{}

\begin{document}
\maketitle

\begin{abstract}
An explicit solution to the Christoffel--Minkowski problem for convex bodies of revolution is presented. The conditions on the prescribed measure involve only first moments over spherical caps, and the support function of the resulting convex body is given by an explicit representation formula in terms of the measure. More generally, existence problems for mixed area measures are addressed. The approach relies on constructing explicit convex solutions to mixed Monge-Amp\`ere equations on $\Rn$ under the assumption of radial symmetry, with the conditions on the measure being expressed through its values on open balls. As a special case, the Dirichlet problem for $k$-Hessian equations on $\Rn$ is treated.

\end{abstract}

\goodbreak

\normalsize
\noindent
\section{Introduction} 
One of the most natural and classical questions in convexity is the Minkowski problem, a central pillar of the field with close connections to numerous fundamental research directions. In its smooth version, it asks for necessary and sufficient conditions on a positive function on the $n$-dimensional unit sphere $\sN$ so that it arises as the inverse of the Gauss curvature (equivalently, the product of the principal radii of curvature) of an $n$-dimensional closed, convex hypersurface.

While the Minkowski problem was resolved in both its smooth and non-smooth versions almost a century ago (cf.\ \cite[Section 8]{SchneiderConvexBodiesBrunn2013}), the more general Christoffel--Minkowski problem, asking the analogous question for the other elementary symmetric functions of the principal radii of curvature, turns out to be much more challenging. For the mean radius of curvature, this is also known as the Christoffel problem, as it was solved in the three-dimensional case by Christoffel \cite{Christoffel}, with later extensions to higher dimensions by Firey \cite{Firey_Christoffel} and Berg \cite{Berg_Christoffel}. For the general problem, a breakthrough was obtained in the early 2000s in the pioneering work of Guan and Ma \cite{Guan_Ma_I}, together with its sequels \cite{Guan_Lin_Ma_II, Guan_Ma_Zhou_III}. There, sufficient conditions for the existence of a solution are provided when the given data is sufficiently regular. See also \cite{bryan_ivaki_scheuer_CM_flows,IvakiMilman,Li_Wan_Wang_21} for subsequent progress. Decades before this, Firey \cite{Firey_revolution} already treated the case when the given function on the sphere is invariant under rotations around a fixed axis. In a major recent development, Brauner, Hofst\"atter, and Ortega-Moreno \cite{BraunerHofstaetterOrtegaMoreno_mixedCM} removed the regularity requirements, thereby providing a complete solution to the Christoffel--Minkowski problem for bodies of revolution by employing methods from valuation theory and integral geometry.

\medskip

In this article, we present an independent solution to the Christoffel--Minkowski problem under the same symmetry assumptions used by Firey and without any further regularity requirements. Our approach yields new and very natural necessary and sufficient conditions on a given measure on $\sN$ for it to be the $j$th area measure of a convex body of revolution, which are formulated differently compared to those obtained in \cite{BraunerHofstaetterOrtegaMoreno_mixedCM}. Most notably, our method is constructive and provides an explicit description of the convex body in terms of the given measure. In particular, we do not resort to the classical solution of the Minkowski problem, and our method also directly shows the uniqueness of the solution.

While our proof is also informed by recent insights from valuation theory and integral geometry, it ultimately relies on finding a convex solution to an entire Monge--Amp\`ere-type equation on $\Rn$. As a byproduct of this strategy, we provide explicit convex solutions for a large class of Monge--Amp\`ere-type equations under the assumption of radial symmetry. 

\subsection{Christoffel--Minkowski Problem}
We denote by $\KN$ the convex bodies in $\RN$, meaning non-empty, compact, convex sets, and we will always assume that $n\geq 2$. To each convex body $K$ we associate its surface area measure $S_n(K,\cdot)$, which is a finite Borel measure on $\sN$. When $K$ has positive volume, this measure is given by
\[
S_n(K,\omega)=\hm^n\big(\{x\in \bd(K) : K \text{ has an outer unit normal in } \omega \text{ at } x \} \big)
\]
for every Borel set $\omega\subseteq\sN$, where $\hm^{n}$ denotes the $n$-dimensional Hausdorff measure and $\bd(K)$ is the boundary of $K$. By Minkowski's existence and uniqueness theorem \cite[Theorem 8.1.1 and Theorem 8.2.2]{SchneiderConvexBodiesBrunn2013}, a given Borel measure $\mu$ on $\sN$ is the surface area measure of a convex body $K\in\KN$ with positive volume, which is unique up to translations, if and only if $\mu$ is centered and not concentrated on any great subsphere.

Generalizing the surface area measure, we consider the $j$th area measures $S_j(K,\cdot)$ with $j\in\{0,\ldots,n\}$, which are defined through the relation
\[
S_n(K+B_r^{n+1},\cdot) = \sum_{j=0}^n \binom{n}{j} r^{n-j} S_j(K,\cdot)
\]
for $K\in\KN$, where $B_r^{n+1}$ is the closed Euclidean ball of radius $r\geq 0$ in $\RN$, and $K+L=\{x+y: x\in K, y\in L\}$ denotes the Minkowski sum of convex bodies. Generalizing the Minkowski problem, the Christoffel--Minkowski problem asks for necessary and sufficient conditions on a Borel measure on $\sN$ such that it appears as the $j$th area measure of a convex body. See, for example, \cite{BoeroeczkyFigalliRamos,Huang_YZ_survey}.

\medskip

Let $\{e_1,\ldots,e_{n+1}\}$ be the standard orthonormal basis of $\RN$ and let $\SO(n)\subset \SO(n+1)$ be the group of rotations $\vartheta$ such that $\vartheta e_{n+1}=e_{n+1}$. A measure $\mu$ on $\sN$ is called $\SO(n)$ invariant if $\mu(\vartheta B)=\mu(B)$ for every Borel set $B\subseteq \sN$ and every $\vartheta\in\SO(n)$. Furthermore, we say that $K\in\KN$ is a body of revolution if $\vartheta K=K$ for every $\vartheta\in\SO(n)$, and we note that the $j$th area measures of such bodies are $\SO(n)$ invariant. Our main result below resolves the Christoffel--Minkowski problem for $\SO(n)$ invariant measures/bodies of revolution. On top of that, we reconstruct the body $K$ from the given measure, using the fact that every convex body $K\in\KN$ is uniquely determined by its support function $h_K(x)=\max_{y\in\RN} \langle x,y\rangle$, $x\in\RN$, where $\langle \cdot,\cdot \rangle$ denotes the usual inner product.

Observe that the orthogonal projection of a body of revolution $K$ onto $e_{n+1}^\perp$ is a (possibly degenerate) ball, whose radius we denote by $R_K\geq 0$ throughout the following. Furthermore, we write $z=(z_1,\ldots,z_{n+1})$ for $z\in\sN$ and consider the three disjoint parts $\sN_-$, $\sN_o$, and $\sN_+$ of the unit sphere $\sN$, where
\[
\sN_{\pm}=\{z\in\sN : \pm z_{n+1} > 0\} \quad \text{and} \quad \sN_o=\{z\in\sN : z_{n+1} = 0\}.
\]
Moreover, for $\alpha\in (0,\frac{\pi}{2}]$ let
\[
C_{\alpha}^{\pm}=\{z\in\sN : \pm z_{n+1} > \cos(\alpha)\}
\]
denote the open spherical caps around $e_{n+1}$ and $-e_{n+1}$ respectively. In addition, we write $\kappa_n$ for the $n$-dimensional volume of $B_1^n$. Lastly, we say that a real-valued function is non-trivial if it is not constant zero.

\goodbreak

\begin{theorem}
\label{thm:main_CM_revol}
Let $j\in\{1,\ldots,n\}$ and let $\mu$ be a finite, centered, $\SO(n)$ invariant Borel measure on $\sN$. There exists a convex body of revolution $K\in\KN$ with $R_K>0$ such that
\[
S_j(K,\cdot)=\mu,
\]
if and only if the functions
\[
F_\mu^+(\alpha) = \frac{\int_{C_\alpha^+} |z_{n+1}|\d \mu(z)}{\sin(\alpha)^{n-j}}\quad \text{and}\quad F_\mu^-(\alpha) = \frac{\int_{C_\alpha^-} |z_{n+1}|\d \mu(z)}{\sin(\alpha)^{n-j}},\quad \alpha\in (0,\tfrac{\pi}{2}],
\]
are non-trivial and non-decreasing. The body $K$ is unique up to translations along $\ospan\{e_{n+1}\}$ and its support function, apart from the addition of $\langle \tau\, e_{n+1},z\rangle$ with some $\tau\in\R$, is given by
\[
h_K(z)=\begin{cases} \cos(\alpha_z) \int_0^{\alpha_z}\frac{1}{\cos(s)^2}\left(\tfrac{1}{\kappa_n} F_\mu^-(s) \right)^{\frac 1j} \d s,\qquad &\text{if } z\in \sN_-,\\
\left(\tfrac{1}{\kappa_n} F_\mu^-\left(\tfrac{\pi}{2}\right) \right)^{\frac 1j},\qquad &\text{if } z\in \sN_o,\\
\cos(\alpha_z)\left( \int_0^{\alpha_z}\frac{1}{\cos(s)^2}\left(\tfrac{1}{\kappa_n} F_\mu^+(s) \right)^{\frac 1j} \d s+c_\mu \right),\qquad &\text{if } z\in \sN_+,
\end{cases}
\]
where $\alpha_z\in [0,\tfrac{\pi}{2}]$ is such that $|z_{n+1}|=\cos(\alpha_z)$, and where
\begin{align*}
c_\mu &= \frac{\mu(\sN_o)}{j\kappa_n \left(\tfrac{1}{\kappa_n} F_\mu^-(\tfrac{\pi}{2})\right)^{\frac{j-1}{j}}} +\sup\nolimits_{\beta\in[0,\frac{\pi}{2})} \left(\tan(\beta) \left(\tfrac{1}{\kappa_n} F_\mu^-\left(\tfrac{\pi}{2}\right) \right)^{\frac 1j} - \int_0^\beta \frac{1}{\cos(s)^2}\left(\tfrac{1}{\kappa_n} F_\mu^-(s) \right)^{\frac 1j} \d s\right)\\
&\quad + \sup\nolimits_{\gamma\in[0,\frac{\pi}{2})} \left(\tan(\gamma) \left(\tfrac{1}{\kappa_n} F_\mu^+\left(\tfrac{\pi}{2}\right) \right)^{\frac 1j} - \int_0^\gamma \frac{1}{\cos(s)^2}\left(\tfrac{1}{\kappa_n} F_\mu^+(s) \right)^{\frac 1j} \d s\right).
\end{align*}
\end{theorem}
Let us give some intuition for the conditions on $\mu$ in Theorem~\ref{thm:main_CM_revol}. The finiteness of $\mu$ ensures that the solution is bounded, and the monotonicity of $F_\mu^\pm$ allows us to find convex functions that describe the lower and upper boundaries of $K$ (cf.\ Section~\ref{se:mm_sc}). Because $\mu$ is centered, we can essentially glue these parts together, possibly with a right cylinder of the form $B_{R}^n+[o,\ell\,e_{n+1}]$ in between. Finally, we have to exclude the case that the functions $F_\mu^\pm$ are trivial, which can only occur if $K$ is a line segment in $\ospan\{e_{n+1}\}$, and in which case the $j$th area measure vanishes identically. Note that our result permits the case when $K$ is a disk, meaning $K=B_R^{n+1}\cap e_{n+1}^\perp$ for some $R>0$. Lastly, we mention that $c_\mu$ is the length of the orthogonal projection of $K$ onto $\ospan\{e_{n+1}\}$.

\medskip

Our proof is based on explicit solutions to entire Monge--Amp\`ere-type equations in $\Rn$, meaning that the domain of the problem is the whole space (see Theorem~\ref{thm:mixed_hessian_unbounded}). Since we establish these for more general mixed Monge--Amp\`ere measures, it is straightforward to generalize our main result, Theorem~\ref{thm:main_CM_revol}, to a large class of existence and uniqueness problems for mixed area measures. We will explain this, together with examples and further comments, in Section~\ref{se:discussion}. Let us point out that necessary and sufficient conditions for such mixed Christoffel--Minkowski problems for bodies of revolution were recently obtained in \cite[Theorem 6.11]{BraunerHofstaetterOrtegaMoreno_mixedCM}, relying on mixed spherical projections \cite{BraunerHofstaetterOrtegaMoreno_Lefscheftz}.

\subsection{Hessian Equations}
The Christoffel--Minkowski problem corresponds to solving a $k$-Hessian equation on the sphere. However, as a by-product of our approach to Theorem~\ref{thm:main_CM_revol}, we obtain an explicit convex solution to the Dirichlet problem for $k$-Hessian equations on $\Rn$, assuming radial symmetry. Consider the open Euclidean ball $D^n_R$ of radius $R>0$ in $\Rn$ and, for simplicity, a positive radially symmetric function $f \in C(D^n_R)$ (later, $f(x)\d x$ will be replaced by a rotationally invariant Borel measure). For $k \in \{1,\dots,n\}$, the problem at hand consists in finding necessary and sufficient conditions on $f$ such that the $k$-Hessian equation 
\begin{equation}\label{eq:smooth_k-Hess}
    \begin{cases}
        [\Hess u(x)]_k=f(x) & \text{ in }D^n_R,\\
        u\equiv 0 &\text{ on }\bd (D^n_R),
    \end{cases}
\end{equation}
admits a convex solution $u$. Here, $[\Hess u(x)]_k$ is the $k$th elementary symmetric function of the eigenvalues of the Hessian matrix $\Hess u(x)$, and we assume that $u$ continuously extends to $\bd(D_R^n)$, the boundary of $D_R^n$.

In their general form, $k$-Hessian equations have been the focus of intense study. See, for example, the work of Caffarelli, Nirenberg, and Spruck \cite{Caff_Nir_Spruck} and the successive investigations by Trudinger \cite{Trudinger_regular_case, Trudinger_weak_solutions} and Trudinger and Wang \cite{Hessian_measures_I, Hessian_measures_II, Hessian_measures_III}. In the latter series of works, \eqref{eq:smooth_k-Hess} was reformulated in a weak sense through the notion of Hessian measures $\Phi_k(u;\cdot)$, $k \in \{ 1,\dots,n\}$, which are non-smooth generalizations of the elementary symmetric functions (see Section~\ref{se:mixed_MA_measures}). In particular, if $u \in C^2(D_R^n)$, then \[\d \Phi_k(u;x)=[\Hess u(x)]_k \d x. \] To complete this picture, the characterization of convex solutions for $k \in \{1,\dots,n-1\}$ is still an open problem which has been treated in a handful of cases only. See, for example, \cite{Ma_Xu_R3, Liu_Ma_Xu_R3} and the works of Salani \cite{Salani_large_solutions, Salani_R3}. The symmetry assumptions we impose below lead to a sensible simplification of the problem, which allows us not only to formulate very simple necessary and sufficient conditions for the existence of convex solutions, but also an explicit formula for the latter.

\medskip

Let $|x|=\sqrt{\langle x,x\rangle }$ denote the Euclidean norm of $x\in\Rn$. Under the assumption of rotational symmetry (meaning that $\mu(\vartheta B)=\mu(B)$ for every $\vartheta\in\SO(n)$ and $B\subseteq D_R^n$ Borel), we provide explicit convex solutions to $k$-Hessian equations when the domain is an open Euclidean ball. We require no further regularity. 

\goodbreak

\begin{theorem}
\label{thm:Hess_intro}
  Let $\mu$ be a finite and rotationally invariant Borel measure on $D_R^n$ for some $R>0$. For $k \in \{1,\ldots,n\}$, there exists a unique radially symmetric convex function $u\colon D^n_R \to \R$, solving the problem
    \begin{equation}
    \label{eq:dirichlet_problem_hessian}
        \begin{cases}
            \Phi_k(u;\cdot)=\mu &\text{ in } D_R^n, \\
            u\equiv 0  &\text{ on } \bd(D_R^n),
        \end{cases}
    \end{equation}
    if and only if
    \[
    r\mapsto \frac{\mu(D_r^n)}{r^{n-k}},\quad r\in(0,R],
    \]
    is non-decreasing. In this case, $u$ is given by
    \begin{equation}
    u(x) = -\int_{|x|}^R \left(\frac{\mu(D_r^n)}{\binom{n}{k}\kappa_n r^{n-k}}\right)^{\frac 1k} \d r
    \end{equation}
    for $x\in D_R^n$.
\end{theorem}
\goodbreak
\noindent
We point out that our approach also leads to the solution of more general problems for mixed Monge--Amp\`ere measures (see Theorem~\ref{thm:mixed_hessian_bounded} below) from which Theorem~\ref{thm:Hess_intro} stems as a simple corollary.

\medskip

The tools developed to prove Theorem~\ref{thm:Hess_intro} adapt as well to solve mixed Monge--Amp\`ere equations of entire type (Theorem~\ref{thm:mixed_hessian_unbounded}), which is the real intermediate step in our approach toward the Christoffel--Minkowski problem.
For some classic results on the topic, see, for example, Bakelman's book \cite{Bakelman}. See also \cite{Wang_Minkowski} for a full solution of the smooth and strictly convex case and \cite{Ulivelli_Entire_MA} for recent developments concerning the connection between Minkowski problems for convex bodies and entire Monge--Amp\`ere equations. 

\section{Preliminaries}
We will mostly work in the $n$-dimensional and $(n+1)$-dimensional Euclidean spaces $\Rn$ and $\RN$, always assuming that $n\geq 2$, and writing $o$ for the origin. For $k\in\{1,\ldots,n-1\}$ we denote $\R^k=\ospan\{e_1,\ldots,e_k\}\subset \Rn$.
In addition, we write $\vol_k$ for the $k$-dimensional volume in $\R^k$. Lastly, for convenience we will consider $D_R^n$ with $R\in(0,\infty]$, where $D_\infty^n=\Rn$. In this context, for $0<r\leq R$, the case $r=R$ will only be allowed when $R$ is finite. Standard references for convex bodies and convex functions are the books by Schneider \cite{SchneiderConvexBodiesBrunn2013} and Rockafellar \cite{RockafellarConvex1997}, respectively. In addition, we refer to Figalli's book \cite{FigalliMongeAmpereequation2017} for a comprehensive treatment of Monge--Amp\`ere equations.

\subsection{Convex Functions}
For a non-empty, open, convex $\Omega\subseteq\Rn$ let
\[
\fconvOmega=\{u\colon \Omega\to\R : u \text{ is convex}\}
\]
denote the set of finite-valued convex functions on $\Omega$. Given $u\in\fconvOmega$, we consider the subdifferential of $u$ at $x\in\Omega$,
\[
\partial u(x)=\{p\in\Rn : u(z)\geq u(x)+\langle p,z-x\rangle \text{ for all } z\in\Omega\},
\]
which is a closed, convex set and generalizes the usual gradient map. We need the following standard result, which can be found, for example, in \cite[Theorem 23.8]{RockafellarConvex1997}, and where we use Minkowski addition of sets.

\begin{lemma}
\label{le:subdiff_sum}
If $u_1,u_2\in\fconvOmega$, then
\[
\partial(u_1+u_2)(x)=\partial u_1(x)+\partial u_2(x)
\]
for every $x\in\Omega$.
\end{lemma}

Next, for a convex function $u\colon\Rn\to(-\infty,\infty]$ let
\begin{equation}
\label{eq:def_conjugate}
u^*(x)=\sup\nolimits_{y\in\Rn} \big(\langle x,y\rangle - u(y)\big)
\end{equation}
be the convex conjugate or Legendre--Fenchel transform of $u$, which is a lower semicontinuous, convex function on $\Rn$. In particular, the convex conjugates of the elements in $\Conv(\Rn)$ are lower semicontinuous, convex functions $w\colon\Rn\to(-\infty,\infty]$ that are proper, meaning $w(x)<\infty$ for some $x\in\Rn$, and super-coercive, that is, $\lim\nolimits_{|x|\to\infty}\frac{w(x)}{|x|}=\infty$. Among these functions, we are particularly interested in those that have compact domain
\[
\dom(w)=\{x\in\Rn : w(x)<\infty\}.
\]
We denote the set of such functions by $\fconvcd$ and remark that $w\in\fconvcd$ implies $w^*\in \Conv(\Rn)$ but $u\in\Conv(\Rn)$ does not necessarily imply that $u^*\in\fconvcd$.
More generally, we have the following relation between subdifferentials and convex conjugates. See, for example, \cite[Theorem 23.5]{RockafellarConvex1997}.

\begin{lemma}
\label{le:subdiff_conjugate}
Let $u\colon \Rn\to(-\infty,\infty]$ be a proper, lower semicontinuous, convex function. For any $x,y\in\Rn$ we have $y\in\partial u(x)$ if and only if $x\in\partial u^*(y)$.
\end{lemma}

We will also need the following result on the composition of convex conjugation with restriction of functions to subspaces, which can be found, for example, in \cite[Theorem 16.3]{RockafellarConvex1997}. We will only need to consider this for functions defined on the whole domain, whose conjugates are, as noted above, super-coercive. For such a lower semicontinuous, super-coercive, convex function $w$ we denote by $\proj_{\R^j} w\colon \R^j\to (-\infty,\infty]$ its projection onto $\R^j$ with $j\in\{1,\ldots,n-1\}$, which is given by
\[
(\proj_{\R^j} w)(x) = \min\nolimits_{z\in (\R^j)^\perp} w(x+z)
\]
for $x\in\R^j$. Alternatively, $\proj_{\R^j} w$ is obtained by projecting the level sets of $w$ onto $\R^j$.

\begin{lemma}
\label{le:conjugate_projection}
Let $j\in\{1,\ldots,n-1\}$. If $u\in\Conv(\Rn)$, then
\[
\big(u\vert_{\R^j}\big)^* = \proj_{\R^j}(u^*),
\]
where on the left side, convex conjugation is understood with respect to the ambient space $\R^j$.
\end{lemma}

\subsection{Mixed Monge--Amp\`ere Measures}
\label{se:mixed_MA_measures}
The Monge--Amp\`ere measure associated to $u\in\fconvOmega$, whose definition goes back to Aleksandrov \cite{Aleksandrov}, is given by
\[
\MA(u;B)=\operatorname{vol}_n\left(\bigcup_{x\in B} \partial u(x) \right) 
\]
for Borel sets $B\subseteq \Omega$. This defines a Radon measure on $\Omega$, which means that the measure is finite on compact subsets of $\Omega$. When $u$ is of class $C^2$, then we obtain the more classical form
\[
\d\MA(u;x)=\det(\Hess u(x)) \d x.
\]
See, for example, \cite[Section 2.1]{FigalliMongeAmpereequation2017}.

Next, let
\[
q(x)=\frac{|x|^2}{2}
\]
for $x\in\Rn$. We define the Hessian measures $\Phi_k(u;\cdot)$, $k\in\{0,\ldots,n\}$, of $u\in\fconvOmega$ through the relation
\[
\MA(q+\lambda u;B) = \vol_n\big(\{x+ \lambda p : x\in B, p\in\partial u(x)\} \big) = \sum_{k=0}^n \lambda^k \Phi_k(u;B)
\]
for $\lambda\geq 0$ and Borel sets $B\subseteq \Omega$. For $k=n$ we retrieve the usual Monge--Amp\`ere measure of $u$. See, for example,
\cite{ColesantiEtAlHessianValuations2020,Trudinger_regular_case,Colesanti_Salani_II,Hessian_measures_I,Hessian_measures_II,ColesantiHug2005}.

Generalizing the above, we consider the polarization of the Monge--Amp\`ere measure with respect to pointwise addition of functions to obtain mixed Monge--Amp\`ere measures (see, for example, \cite{ColesantiEtAlHadwigertheoremconvex2022}), which also appear as mixed Hessian measures in \cite{Hessian_measures_III}. More precisely, we assign to every $n$-tuple of functions $u_1,\ldots,u_n\in\fconvOmega$ the unique Radon measure $\MA(u_1,\ldots,u_n;\cdot)$ that is symmetric in its entries and such that
\[
\MA(\lambda_1 u_1+\cdots+\lambda_m u_m;B)=\sum_{i_1,\ldots,i_n=1}^m \lambda_{i_1}\cdots \lambda_{i_n} \MA(u_{i_1},\ldots,u_{i_n};B)
\]
for every $m\in\N$, $u_1,\ldots,u_m\in\fconvOmega$, $\lambda_1,\ldots,\lambda_m\geq 0$, and $B\subseteq \Omega$ Borel. The family of Hessian measures appears as the special case
\begin{equation}
\label{eq:phi_k_ma}
\Phi_k(u;\cdot)=\binom{n}{k} \MA(u[k],q[n-k];\cdot)
\end{equation}
for $u\in\fconvOmega$ and $k\in\{0,\ldots,n\}$, where $u[k]$ means that we repeat the entry $u$ in the mixed Monge--Amp\`ere measure $k$ times, and similarly $q$ is repeated $n-k$ times. In addition, the measures
\begin{equation}
\label{eq:def_MAk}
\MA_k(u;\cdot)=\MA(u[k],|\cdot|[n-k];\cdot),
\end{equation}
with $u\in\fconvOmega$ and $k\in\{0,\ldots,n\}$ are of particular interest to us and have recently been used extensively in the theory of valuations on convex functions (see, for example, 
\cite{ColesantiEtAlHadwigertheoremconvex2022,HugMussnigUlivelli2,MouamineMussnig2}). Among others, we have the following Kubota-type formula, derived by the authors together with Hug in \cite[Theorem 1.3]{HugMussnigUlivelli1}.
Note that \cite[Theorem 1.3]{HugMussnigUlivelli1} is formulated in terms of integration over the Grassmannian of $k$-dimensional linear subspaces, which can be rewritten as an integral over $\SO(n)$ (w.r.t.\ the Haar probability measure) that is used here.

\begin{theorem}
\label{thm:ck_ma_SOn}
Let $k\in\{1,\ldots,n\}$. If $\varphi\colon\Rn\to[0,\infty)$ is measurable, then
\[
\frac{1}{\kappa_n}\int_{\Rn}\varphi(x)\d\MA_k(u;x) = \frac{1}{\kappa_k} \int_{\SO(n)} \int_{\vartheta \R^k} \varphi(y) \d \MA'(u\vert_{\vartheta \R^k}, y) \d\vartheta
\]
for every $u\in\Conv(\Rn)$, where $\MA'$ is the usual Monge--Amp\`ere measure in the ambient space $\vartheta \R^k$.
\end{theorem}

We close this section by describing the Monge--Amp\`ere measure in terms of the convex conjugate of a function. For $u\in\fconvOmega$ and continuous function $\varphi\colon\Omega\to\R$ with compact support, we have
\begin{equation}
\label{eq:conjugate_MA}
\int_{\Rn} \varphi(x)\d\MA(u;x)=\int_{\dom(u^*)} \varphi(\nabla u^*(x)) \d x,
\end{equation}
where we have used that a convex function is differentiable almost everywhere in the interior of its domain. See, for example, \cite[Section 10.4]{ColesantiEtAlHessianValuations2020}. See also \cite[Theorem 5.1]{ColesantiEtAlHadwigertheoremconvex2022}.

\subsection{Mixed Area Measures}
Analogous to mixed Monge--Amp\`ere measures, to every $n$-tuple of bodies $K_1,\ldots,K_n\in\KN$ we assign the mixed area measure $S(K_1,\ldots,K_n,\cdot)$, which is determined by being symmetric in its entries and arising from
\[
S(\lambda_1 K_1+\cdots + \lambda_m K_m,\omega) \sum_{i_1,\ldots,i_n=1}^m \lambda_{i_1}\cdots \lambda_{i_n} S(K_{i_1},\ldots,K_{i_n},\omega)
\]
for every $m\in\N$, $K_1,\ldots,K_m\in\KN$, $\lambda_1,\ldots,\lambda_m\geq 0$, and $\omega\subseteq \sN$ Borel. For $j\in\{0,\ldots,n\}$, the $j$th area measure appears as the special case
\begin{equation}
\label{eq:sj_mixed_relation}
S_j(K,\cdot) = S(K[j],B_1^{n+1}[n-j],\cdot)
\end{equation}
for $K\in\KN$. Similar to this, we define the measures
\begin{equation}
\label{eq:bar_s_j}
\bar{S}_j(K,\cdot)=S(K[j],B_1^{n+1}\cap e_n^\perp [n-j],\cdot)
\end{equation}
for $K\in\KN$, where compared to \eqref{eq:sj_mixed_relation} we have replaced the $(n+1)$-dimensional Euclidean unit ball with the $n$-dimensional unit disk. We will see in Lemma~\ref{le:int_uk} that these measures are closely related to the measures $\MA_j(u;\cdot)$, which were defined in \eqref{eq:def_MAk}.

\medskip

The distinctive feature of the measures $\bar{S}_j(K,\cdot)$ is that they are ideal tools for investigating bodies of revolution. This is encoded in the following Kubota-type formula, which, as Theorem~\ref{thm:ck_ma_SOn} above, was obtained by the authors together with Hug in \cite[Theorem 3.2]{HugMussnigUlivelli1}. We state a slightly simplified version that meets the requirements of the present article and note that $\proj_{\bar{E}_{j+1}}$ is the orthogonal projection onto $\bar{E}_{j+1}=\ospan\{e_1,\ldots,e_j,e_{n+1}\}$.

\begin{theorem}
\label{thm:bar_sk_rewrite}
If $j\in\{1,\ldots,n\}$, then
\begin{equation*}
\bar{S}_j(K,\omega) = \frac{\kappa_{n}}{\kappa_j} \int_{\SO(n)} S_j'(\proj_{\vartheta \bar{E}_{j+1}} K,\omega \cap \vartheta \bar{E}_{j+1}) \d\vartheta
\end{equation*}
for every $K \in \KN$ and $\omega\subseteq\sN$ Borel, where $S_j'$ is the usual surface area measure in the ambient space $\vartheta \bar{E}_{j+1}$.
\end{theorem}

\subsection{Measure Theory}
We say that a Borel set $B\subseteq \Rn$ is rotationally invariant if $x\in B$ implies that $\vartheta x\in B$ for every $\vartheta\in\SO(n)$. In addition, $\chi_B$ will denote the usual characteristic function of $B\subseteq \Rn$.

\begin{lemma}
\label{le:rot_inv_meas_determined}
If $\mu_1,\mu_2$ are rotationally invariant Radon measures on $\Rn$ such that
\begin{equation*}
\mu_1(B)=\mu_2(B)
\end{equation*}
for every rotationally invariant Borel set $B\subseteq\Rn$, then $\mu_1\equiv \mu_2$. In other words, every rotationally invariant Radon measure on $\Rn$ is completely determined by its behaviour on rotationally invariant Borel subsets of $\Rn$. Moreover,
\begin{equation}
\label{eq:mu_1_leq_mu_2_rot}
\mu_1(B)\leq \mu_2(B)
\end{equation}
for rotationally invariant Borel sets $B\subseteq \Rn$ implies that $\mu_1\leq \mu_2$ on $\Rn$.
\end{lemma}
\begin{proof}
Observe first that it is enough to prove the statement regarding inequalities, since $\mu_1(B)=\mu_2(B)$ implies that both $\mu_1(B)\leq \mu_2(B)$ and $\mu_2(B)\leq \mu_1(B)$. Thus, we will assume that \eqref{eq:mu_1_leq_mu_2_rot} holds.

Let $\mu$ be a rotationally invariant Radon measure on $\Rn$. For every Borel set $B \subset \Rn$ it follows from the rotational symmetry of $\mu$ that
\begin{align}
    \begin{split}
    \label{eq:mu_int_chi_vartheta}
    \mu(B) &= \int_{\SO(n)} \mu(\vartheta B) \d \vartheta = \int_{\SO(n)} \int_{\Rn} \chi_{\vartheta B}(x) \d\mu(x) \d\vartheta = \int_{\Rn} \int_{\SO(n)} \chi_{\vartheta B}(x) \d\vartheta \d\mu(x),
    \end{split}
\end{align}
where we have used the Fubini--Tonelli theorem in the last step. To every Borel set $B\subseteq \Rn$ we now associate a function
\[
\Psi_B(x)= \int_{\SO(n)} \chi_{\vartheta B}(x) \d \vartheta,\quad x\in\Rn,
\]
which is a non-negative, measurable function on $\Rn$ by Fubini's theorem. Since $\Psi_B$ is also a radially symmetric function, it easily follows from \eqref{eq:mu_1_leq_mu_2_rot} that
\[
\int_{\Rn} \Psi_B(x) \d\mu_1(x) \leq \int_{\Rn} \Psi_B(x) \d\mu_2(x)
\]
for every Borel set $B\subseteq \Rn$. By \eqref{eq:mu_int_chi_vartheta} this implies $\mu_1(B)\leq\mu_2(B)$, completing the proof.
\end{proof}

Note that the $\sigma$-algebra of rotationally invariant Borel sets is generated by centered, open Euclidean balls. Furthermore, finite disjoint unions of annuli (set differences of the aforementioned balls) form a generating algebra of this $\sigma$-algebra, and positivity of a rotationally invariant Radon measure on annuli implies positivity on all rotationally invariant Borel sets.
Hence, the following corollary is easily obtained from standard arguments.
\begin{corollary}
\label{cor:diks}
If $\mu_1,\mu_2$ are rotationally invariant Radon measures on $\Rn$ such that
\[
\mu_1(D^n_r)=\mu_2(D^n_r)
\]
for every $r>0$, then $\mu_1\equiv \mu_2$. Similarly,
\[
\mu_1(A)\leq \mu_2(A)
\]
for every annulus
\[
A=\{x\in\Rn : s\leq |x| < t\}
\]
with $0\leq s < t$ implies that $\mu_1\leq \mu_2$ on $\Rn$.
\end{corollary}

\section{Mixed Monge--Amp\`ere Equations}
The aim of this section is to prove Theorem \ref{thm:Hess_intro}. The strategy reduces to solving a one-dimen\-sional problem based on a simplified representation of mixed Monge--Amp\`ere measures in the rotationally invariant case. We want to stress the fact that the method provides explicit solutions, and it can be interpreted as an integrated version of the ODE approach appearing in the work of Firey on the Christoffel--Minkowski problem \cite{Firey_revolution} and Salani on large solutions of Hessian equations \cite{Salani_large_solutions}.

\subsection{Radial Symmetry and Mixed Measures}
We turn our attention to mixed Monge--Amp\`ere measures of radially symmetric convex functions on $D_R^n$ for some $R\in(0,\infty]$. First, it follows from the definition of mixed Monge--Amp\`ere measures that
\begin{equation}
\label{eq:mixe_MA_On}
\MA(u_1\circ\vartheta^{-1},\ldots,u_n\circ \vartheta^{-1};B) = \MA(u_1,\ldots,u_n;\vartheta^{-1}B)
\end{equation}
for every $u_1,\ldots,u_n\in\Conv(D_R^n)$, $\vartheta\in\SO(n)$, and $B\subseteq D_R^n$ Borel (cf.\ \cite[Proposition 8.1]{ColesantiEtAlHessianValuations2020}). Thus, if $u_1,\ldots,u_n$ are radially symmetric, then $\MA(u_1,\ldots,u_n)$ is rotationally invariant. 

Next, for a radially symmetric $u\in\Conv(D_R^n)$ let
\begin{equation}
\label{eq:def_p_u}
p_u(r)=\sup\{|p|: p\in\partial u(x), 0\leq |x|<r\}\in[0,\infty]
\end{equation}
for $0<r\leq R$, where we recall that we consider $r=R$ only when $R$ is finite. Observe that we can write $u(x)=\bar{u}(|x|)$ for $x\in D_R^n$, where $\bar{u}\colon [0,R)\to\R$ is the (non-decreasing), convex profile function of $u$. Furthermore, for $x\neq o$, we have $p\in \partial \bar{u}(|x|)$ if and only if $p \frac{x}{|x|}\in \partial u(x)$ (see, for example, \cite[Lemma 5.3]{HugMussnigUlivelli1}). Since
\[
\partial \bar{u}(r)=\{p\in\R : \bar{u}'_-(r)\leq p \leq \bar{u}'_+(x)\}
\]
for $r\in(0,R)$, where $\bar{u}'_-$ and $\bar{u}'_+$ denote the left and right derivatives of $\bar{u}$ respectively, it follows from \cite[Theorem 24.1]{RockafellarConvex1997} that
\begin{equation}
\label{eq:p_u_left_derviative}
p_u(r)=\sup\{\bar{u}'_+(s) : 0<s<r\}=\bar{u}'_-(r)
\end{equation}
for every $r\in(0,R)$, and $p_u(R)=\lim_{r\to R^-} \bar{u}'_-(r)$. In particular, $p_u(r)$ is non-negative, continuous from the left, and non-decreasing.

\begin{lemma}
\label{le:ma_d_r_n}
Let $R\in(0,\infty]$. If $u_1,\ldots,u_n\in\Conv(D_R^n)$ are radially symmetric, then
\[
\MA(u_1,\ldots,u_n;D_r^n) = \kappa_n p_{u_1}(r)\cdots p_{u_n}(r)
\]
for every $0<r\leq R$, with the convention $0\cdot\infty=0$.
\end{lemma}
\begin{proof}
Let $0<r\leq R$ be fixed and assume first that $p_{u_1}(r),\ldots,p_{u_n}(r)<\infty$. By definition, we have
\begin{equation}
\label{eq:ma_expansion_proof_d_r_n}
\MA(\lambda_1 u_1+\cdots + \lambda_n u_n;D_r^n) = \sum_{i_1,\ldots,i_n=1}^n \lambda_{i_1}\cdots \lambda_{i_n}\MA(u_{i_1},\ldots,u_{i_n};D_r^n)
\end{equation}
for $\lambda_1,\ldots,\lambda_n\geq 0$. On the other hand, by Lemma~\ref{le:subdiff_sum},
\begin{align*}
\MA(\lambda_1 u_1+\cdots + \lambda_n u_n;D_r^n) &= \vol_n\big(\{ p \in \partial(\lambda_1 u_1 + \cdots + \lambda_n u_n)(x), x \in D_r^n\} \big)\\
&= \vol_n\big(\{ p \in \lambda_1 \partial u_1(x) + \cdots + \lambda_n \partial u_n (x), x \in D_r^n\} \big).
\end{align*}
Due to the radial symmetry of the functions, the set to be considered in the last expression is, depending on the functions, an open or closed ball of radius $\lambda_1 p_{u_1}(r)+\cdots + \lambda_n p_{u_n}(r)$. Therefore,
\begin{align*}
\MA(\lambda_1 u_1+\cdots + \lambda_n u_n;D_r^n) &= \vol_n\big((\lambda_1 p_{u_1}(r) + \cdots + \lambda_n p_{u_n}(r)) B^n\big)\\
&= \kappa_n (\lambda_1 p_{v_1}(r) + \cdots + \lambda_n p_{v_n}(r))^n\\
&= \kappa_n \sum_{i_1,\ldots,i_n=1}^n \lambda_{i_1} \cdots \lambda_{i_n} p_{u_{i_1}}(r) \cdots p_{u_{i_n}}(r)
\end{align*}
for $0<r\leq R$ and $\lambda_1,\ldots,\lambda_n\geq 0$, and the statement follows after comparing coefficients with \eqref{eq:ma_expansion_proof_d_r_n}.

Now assume that $p_{u_i}(r)=\infty$ for some $i\in\{1,\ldots,n\}$, which can only occur when $r=R<\infty$. As $\MA(u_1,\ldots,u_n;\cdot)$ is a measure, it is continuous from below with respect to nested open sets and therefore
\[
\MA(u_1,\ldots,u_n;D_R^n)=\lim\nolimits_{t\to R^-} \MA(u_1,\ldots,u_n;D_t^n).
\]
The result now follows from the previous case.
\end{proof}

Observe that for radially symmetric $u\in\Conv(D_R^n)$ we have $\lim_{r\to 0^+} \bar{u}'_-(r)=\bar{u}'_+(0)$, which corresponds to the fact that
\[
\lim\nolimits_{r\to 0^+} \MA(u_1,\ldots,u_n;D_r^n)=\MA(u_1,\ldots,u_n;\{o\}).
\]

Let us briefly highlight some consequences of Lemma~\ref{le:ma_d_r_n}. For radially symmetric\linebreak$u\in\Conv(D_R^n)$ and $k\in\{0,\ldots,n\}$ we choose $u_1=\cdots=u_k=u$ and $u_{k+1}=\cdots=u_n=q$ in Lemma~\ref{le:ma_d_r_n}. Since $p_q(r)=r$ for $r>0$, we obtain
\begin{equation}
\label{eq:Hess_radial}
    \Phi_k(u;D_r^n) = \binom{n}{k}\MA(u[k],q[n-k];D_r^n) = \binom{n}{k} \kappa_n p_u(r)^k r^{n-k}
\end{equation}
for $0<r\leq R$, where we have used \eqref{eq:phi_k_ma}. Similarly, since $\nabla |x|= \frac{x}{|x|}$ for $x\neq o$, choosing $u_{k+1}=\cdots = u_n=|\cdot|$ results in
\begin{equation}
\label{eq:MA_radial}
    \MA_k(u;D_r^n)=\MA(u[k],|\cdot|[n-k];D_r^n)=\kappa_n p_u(r)^k
\end{equation}
for $0<r \leq R$.  By \eqref{eq:Hess_radial} and \eqref{eq:MA_radial} we thus obtain the following result, which encodes the condition on $\mu$ in Theorem~\ref{thm:Hess_intro}.

\begin{lemma}
\label{le:ma_k_phi_k}
Let $k\in\{0,\ldots,n\}$ and $R\in(0,\infty]$. If $u\in\Conv(D_R^n)$ is radially symmetric, then
\[
\binom{n}{k} \MA_k(u;D_r^n) = \frac{\Phi_k(u;D_r^n)}{r^{n-k}}
\]
for every $0<r\leq R$.
\end{lemma}
We note that using an integral transform between the measures $\Phi_k(v;\cdot)$ and $\MA_k(v;\cdot)$ (see for example \cite[Section 5]{HugMussnigUlivelli2}), developed in the context of valuation theory, one can show that Lemma~\ref{le:ma_k_phi_k} also holds for general convex functions without assuming radial symmetry. 

\subsection{Explicit Solutions}
\label{se:explicit_solutions}
In this section, we prove our main results on convex functions. Considering bounded domains first, we establish necessary and sufficient conditions on a rotationally invariant, finite Borel measure $\mu$ on $D_R^n$, $R>0$, such that there exists a radially symmetric $u\in\Conv(D_R^n)$ solving
\begin{equation}
\label{eq:mixed_problem}
\begin{cases}
\MA(u[k],u_{k+1},\ldots,u_n;\cdot)=\mu &\text{ in } D_R^n, \\
	u\equiv 0  &\text{ on } \bd(D_R^n).
\end{cases}
\end{equation}
Here, $k\in\{1,\ldots,n\}$ and, in case $k<n$, $u_{k+1},\ldots,u_n\in \Conv(D_R^n)$ are given radially symmetric functions. In fact, we will even give an explicit description of the function $u$.

\medskip

Our solution of \eqref{eq:mixed_problem} relies on the following consequence of \cite[Theorem 24.1 and Theorem 24.2]{RockafellarConvex1997}, where we have used that the function $G$ in the statement below is left-continuous and finite-valued. In addition, since $G$ is non-negative and non-decreasing, we may continuously extend $G(0)=\lim_{t\to 0^+} G(t)$.
Lastly, for $R=\infty$ (which we will need later), we understand $(0,R]$ and $[0,R]$ as $(0,\infty)$ and $[0,\infty)$ respectively.

\begin{theorem}
\label{thm:rockafellar_integegrate}
Let $R\in (0,\infty]$. If $G\colon (0,R]\to[0,\infty)$ is a left-continuous and non-decreasing function, then
\[
w(r)=\int_0^r G(t)\d t,\quad r\in[0,R],
\]
is a well-defined, finite-valued convex function on $[0,R]$ such that $w'_-\equiv G$ on $(0,R]$. If $\tilde{w}\colon [0,R]\to\R$ is any other convex function such that $\tilde{w}
_-\equiv G$, then there exists $c\in\R$ such that $\tilde{w}\equiv w+c$.
\end{theorem}

Recalling the definition of $p_u$ in \eqref{eq:def_p_u}, Lemma~\ref{le:ma_d_r_n} shows that if in \eqref{eq:mixed_problem} one of the functions $u_i$ with $i\in\{k+1,\ldots,n\}$ is such that $p_{u_i}(R)=\infty$, then $\MA(u[k],u_{k+1},\ldots,u_n;D_R^n)$ is finite if and only if one of the other functions in the mixed Monge--Amp\`ere, say $v\in\{u,u_{k+1},\ldots,u_n\}\setminus \{ u_i\}$, satisfies $p_{v}\equiv 0$. But this means that 
\[
\MA(u[k],u_{k+1},\ldots,u_n;\cdot)\equiv 0,
\]
so only one choice for $\mu$ remains in \eqref{eq:mixed_problem}.
In case $v=u$, then $p_u\equiv 0$ implies that $u$ is constant, which together with the boundary condition means that $u\equiv 0$. In the remaining case, $v$ is one of the other functions $u_i$ and the mixed Monge--Amp\`ere measure vanishes regardless of the choice of $u$. In particular, we lose uniqueness in this case.

Similarly, it is now easy to see that the problem at hand, as posed in \eqref{eq:mixed_problem}, will not have a unique solution if $p_{u_i}(r_o)=0$ for some $i\in\{k+1,\ldots,n\}$ and some $0<r_o\leq R$, since then $p_{u_i}(r)=0$ for every $0<r\leq r_o$ and, in particular, the mixed Monge--Amp\`ere measure vanishes on $D_{r_o}^n$.

\medskip

Considering the above, in our following result, it is reasonable to assume that the functions $p_{u_{k+1}},\ldots,p_{u_n}$ are positive and finite on $(0,R]$. Note that by \eqref{eq:phi_k_ma} the case $u_{k+1}=\cdots=u_n=q$ then yields Theorem~\ref{thm:Hess_intro}, since $p_q(r)=r$ for $r>0$ (cf.\ \eqref{eq:Hess_radial} and Lemma~\ref{le:ma_k_phi_k}). Furthermore, for $k=n$ we consider $F_\mu(r)=\mu(D_r^n)$ in \eqref{eq:mu_d_t_n_condition}, which is trivially non-decreasing as $\mu$ is a measure. This is in line with the classical solution to the Dirichlet problem for the Monge--Amp\`ere measure (see, for example, \cite[Theorem 2.13]{FigalliMongeAmpereequation2017}).

\begin{theorem}
\label{thm:mixed_hessian_bounded}
Let $\mu$ be a finite and rotationally invariant Borel measure on $D_R^n$ for some $R>0$ and let $k\in\{1,\ldots,n\}$. If $u_{k+1},\ldots,u_n\in\Conv(D_R^n)$ are radially symmetric such that
    \[
    0<p_{u_{k+1}},\ldots,p_{u_n}<\infty\qquad \text{on } (0,R],
    \]
    then there exists a unique radially symmetric $u\in\Conv(D_R^n)$, solving the problem
	\begin{equation}
		\label{eq:dirichlet_problem_mixed}
		\begin{cases}
        \MA(u[k],u_{k+1},\ldots,u_n;\cdot)=\mu &\text{ in } D_R^n, \\
			u\equiv 0  &\text{ on } \bd(D_R^n),
		\end{cases}
	\end{equation}
	if and only if
	\begin{equation}
		\label{eq:mu_d_t_n_condition}
		r\mapsto F_\mu(r)= \frac{\mu(D_r^n)}{p_{u_{k+1}}(r)\cdots p_{u_n}(r)},\quad r\in(0,R],
	\end{equation}
	is non-decreasing. In this case, $u$ is given by
    \begin{equation}
    \label{eq:sol_u_dirichlet}
    u(x) = -\int_{|x|}^R \left( \tfrac{1}{\kappa_n} F_\mu(t)\right)^{\frac 1k} \d t
    \end{equation}
    for $x\in\Rn$.
\end{theorem}
\begin{proof}
Assume first that there exists a radially symmetric convex function $u\in \Conv(D_R^n)$ that solves \eqref{eq:dirichlet_problem_mixed}. Denoting by $\bar{u}$ its profile function, it follows from Lemma~\ref{le:ma_d_r_n} and \eqref{eq:p_u_left_derviative} that
	\[
	r\mapsto F_\mu(r)=\frac{\MA(u[k],u_{k+1},\ldots,u_n;D_r^n)}{p_{u_{k+1}}(r)\cdots p_{u_n}(r)}=\kappa_n p_u(r)^k = \kappa_n(\bar{u}'_-(r))^k, \quad r\in (0,R],
	\]
	must be non-decreasing.

    Now let $\mu$ be such that $F_\mu$ is non-decreasing. Since $\mu$ is a finite (positive) measure, the function $F_\mu$ is finite-valued and non-negative. Futhermore, as $D_r^n$ is open, the map $r\mapsto \mu(D_r^n)$ is left-continuous, and thus, by \eqref{eq:p_u_left_derviative}, also $F_\mu$ is left-continuous. Since also $(F_\mu/\kappa_n)^{\frac 1k}$ has these properties, it follows from Theorem~\ref{thm:rockafellar_integegrate} that
    \[
    \bar{u}(r)=-\int_r^R \left(\tfrac{1}{\kappa_n}F_\mu(t)\right)^{\frac 1k} \d t=\int_0^r \left(\tfrac{1}{\kappa_n}F_\mu(t)\right)^{\frac 1k} \d t - \int_0^R \left(\tfrac{1}{\kappa_n}F_\mu(t)\right)^{\frac 1k} \d t,\quad r\in [0,R],
    \]
    is a finite-valued convex function on $[0,R]$ such that $\bar{u}'_-\equiv (F_\mu/\kappa_n)^{\frac 1k}$ on $(0,R]$. Furthermore, since $F_\mu$ is non-negative, $\bar{u}'_-$ is also non-negative, and thus, $\bar{u}$ is the profile function of a radially symmetric convex function $u\in\Conv(D_R^n)$. Clearly, $u$ is given by \eqref{eq:sol_u_dirichlet} and continuously extends to $\bd(D_R^n)$ with value $0$. It follows from Lemma~\ref{le:ma_d_r_n} and \eqref{eq:p_u_left_derviative} that
    \[
    \MA(u[k],u_{k+1},\ldots,u_n;D_r^n) = \kappa_n(\bar{u}'_-(r))^k p_{u_{k+1}}(r)\cdots p_{u_n}(r) = F_\mu(r)p_{u_{k+1}}(r)\cdots p_{u_n}(r) = \mu(D_r^n)
    \]
    for every $r\in (0,R]$. Since due to \eqref{eq:mixe_MA_On} the mixed Monge--Amp\`ere measure in the above inequality is rotationally invariant, it now follows from Corollary~\ref{cor:diks} (after trivially extending the measures to $\Rn$ by setting them equal to $0$ on $\Rn\setminus D_R^n$) that $\MA(u[k],u_{k+1},\ldots,u_n;\cdot)\equiv \mu$, and thus, $u$ solves \eqref{eq:dirichlet_problem_mixed}.

    Lastly, suppose that also $\tilde{v}\colon D_R^n\to\R$ is a radially symmetric convex solution to \eqref{eq:dirichlet_problem_mixed}. Then by the first part of the proof, the profile functions of $v$ and $\tilde{v}$ must have the same left derivatives, which by Theorem~\ref{thm:rockafellar_integegrate} implies that $\tilde{v}\equiv v+c$ for some $c\in\R$. Since $v\equiv 0 \equiv \tilde{v}$ on $\bd(D_R^n)$, we must have $c=0$, and thus, $\tilde{v}\equiv v$. 
\end{proof}
\noindent
Let us remark that Theorem~\ref{thm:mixed_hessian_bounded} also holds for $n=1$ and even measures. Beyond that, it is easy to check that, using again Theorem~\ref{thm:rockafellar_integegrate}, one can obtain an explicit convex solution also without assuming that $\mu$ is even.

\medskip

We close this section with the corresponding result for entire equations, which we shall need later in the proof of Theorem~\ref{thm:main_CM_revol}. We omit the proof, since it is almost identical to that of Theorem~\ref{thm:mixed_hessian_bounded}, only noting that $\mu$ must be a Radon measure so that the solution is finite-valued.

\begin{theorem}
\label{thm:mixed_hessian_unbounded}
	Let $\mu$ be a rotationally invariant Radon measure on $\Rn$ and let $k\in\{1,\ldots,n\}$. If $u_{k+1},\ldots,u_n\in\Conv(\Rn)$ are radially symmetric such that
    \[
    0<p_{u_{k+1}},\ldots,p_{u_n}<\infty\qquad \text{on } (0,\infty),
    \]
    then there exists a unique radially symmetric $u\in\Conv(\Rn)$, solving the problem
	\begin{equation}
		\label{eq:unbounded_problem_mixed}
		\begin{cases}
        \MA(u[k],u_{k+1},\ldots,u_n;\cdot)=\mu &\text{ in } \Rn, \\
			u(o)=0,
		\end{cases}
	\end{equation}
	if and only if
	\begin{equation*}
		r\mapsto F_\mu(r)= \frac{\mu(D_r^n)}{p_{u_{k+1}}(r)\cdots p_{u_n}(r)},\quad r>0,
	\end{equation*}
	is non-decreasing. In this case, $u$ is given by
    \begin{equation*}
    u(x) = \int_{0}^{|x|} \left( \tfrac{1}{\kappa_n}F_\mu(t)\right)^{\frac 1k} \d t
    \end{equation*}
    for $x\in\Rn$.
\end{theorem}

\section{Christoffel--Minkowski Problems}
\label{se:CM_problems}
With Theorem \ref{thm:mixed_hessian_unbounded} at hand, we now prove Theorem \ref{thm:main_CM_revol}, the main result of this paper. The method is based on suitable analytic representations of convex bodies and corresponding area measures, which we now briefly present.

\subsection{Mixed Measures and Spherical Caps}
\label{se:mm_sc}
The following, we adopt the approach presented in \cite[Section 3]{KnoerrUlivelli2023}.
To every convex body $K\in\KN$, we associate the
functions
\[
\lfloor K\rfloor(x)=\begin{cases}
    \min_{(x,t)\in K} t\quad &\text{if } x\in\proj_{\Rn} K,\\
    \infty\quad &\text{else},
\end{cases}
\qquad
\lceil K\rceil(x)=\begin{cases}
    -\max_{(x,t)\in K} t\quad &\text{if } x\in\proj_{\Rn} K,\\
    \infty\quad &\text{else},
\end{cases}
\]
which are lower semicontinuous, convex functions with compact domain on $\Rn$, i.e.\ they are elements of $\fconvcd$. Note that $\dom(\lfloor K \rfloor) = \dom(\lceil K \rceil)=\proj_{\Rn} K$ and that the closure of
\[
\{x\in \bd( K): K \text{ has an outer unit normal in } \sN_- \text{ at } x\}
\]
is the graph of $\lfloor K \rfloor$ over its domain. To illustrate that the equality only holds up to taking the closure, we mention the example $K=B_1^{n+1}$. 
A similar statement holds for the boundary points with outer unit normals in $\sN_+$ and the graph of $\lceil K \rceil$. Observe that $K$ can be uniquely recovered from these two functions via
\begin{equation}
\label{eq:retrieve_K_floor_ceil}
K=\{(x,t)\in\Rn\times \R : \lfloor K \rfloor(x)\leq t \leq -\lceil K \rceil(x)\}.
\end{equation}
Conversely, every two lower semicontinuous, convex functions $w^-,w^+$ that have the same compact domain on $\Rn$ and such that $\max_{x\in\dom(w^-)} w^-(x) \leq - \max_{x\in\dom(w^+)} w^+(x)$ define a unique convex body $K\in\KN$ so that $w^-=\lfloor K \rfloor$ and $w^+=\lceil K \rceil$.

The functions $\lfloor K \rfloor$, $\lceil K \rceil$ are intimately related to the support function of $K$. Denoting $u_K^-(x)=h_K(x,-1)$ for $x\in\Rn$ and, similarly, $u_K^+(x)=h_K(x,1)$, we have
\begin{equation}
\label{eq:floor_conjugate}
(\lfloor K \rfloor)^*(x)=u_K^-(x)\quad \text{and}\quad (\lceil K \rceil)^*(x)=u_K^+(x)
\end{equation}
for $x\in\Rn$ (cf.\ \cite[Equation (9)]{KnoerrUlivelli2023}). We will now relate the mixed area measure of $K_1,\ldots,K_n\in\KN$ to the mixed Monge--Amp\`ere measures of the associated functions $u_{K_1}^{\pm},\ldots,u_{K_n}^{\pm}$ which are elements of $\Conv(\Rn)$. To do so, let $z=(z_1,\ldots,z_{n+1})\in\sN_-$ and denote by
\[
\gnom(z)=\frac{(z_1,\ldots,z_n)}{|z_{n+1}|}\in\Rn
\]
its gnomonic projection, which defines a homeomorphism $\gnom\colon \sN_-\to\Rn$. Similarly, one defines $\gnom$ on $\sN_+$, which is again a homomorphism. We have the following consequence of \cite[Corollary 4.9]{HugMussnigUlivelli1}.

\begin{lemma}
\label{le:int_uk}
If $\varphi\colon \Rn\to[0,\infty)$ is measurable, then
\[
\int_{\Rn} \varphi(x)\d\MA(u_{K_1}^{\pm},\ldots,u_{K_n}^{\pm};x)=\int_{\sN_{\pm}} |z_{n+1}| \varphi(\gnom(z)) \d S(K_1,\ldots,K_n,z)
\]
for every $K_1,\ldots,K_n\in\KN$.
\end{lemma}
Next, let $z\in\sN_{\pm}$ and observe that $|\gnom(z)|=\sqrt{1-|z_{n+1}|^2}/|z_{n+1}|$. Thus,
$|\gnom(z)|<\tan(\alpha)$ if and only if $|z_{n+1}|>\cos(\alpha)$ and, therefore,
\begin{equation}
\label{eq:chi_d_tan}
\chi_{D_{\tan(\alpha)}^n}(\gnom(z))=\chi_{C_{\alpha}^\pm}(z)
\end{equation}
for $z\in \sN_{\pm}$ and $\alpha\in (0,\frac{\pi}{2}]$. Together with Lemma~\ref{le:int_uk} this shows
\[
\int_{C_\alpha^\pm} |z_{n+1}| \d S_j(K,z)=\MA\left(u_K^\pm [j],u_B[n-j];D^n_{\tan(\alpha)}\right)
\]
for every $j\in\{1,\ldots,n\}$, $K\in\KN$, and $\alpha\in(0,\frac{\pi}{2}]$, where
\[
u_B(x)=h_{B_1^{n+1}}(x,\pm 1) = \sqrt{1+|x|^2}
\]
for $x\in\Rn$. If, in addition, $K$ is a body of revolution, then $u_K^\pm$ must be radially symmetric and Lemma~\ref{le:ma_d_r_n} shows that
\begin{equation}
\label{eq:int_moment_cap}
\int_{C_\alpha^\pm} |z_{n+1}| \d S_j(K,z)=\kappa_n p_{u_K^\pm}(\tan(\alpha))^j \left(\frac{\tan(\alpha)}{\sqrt{1+\tan^2(\alpha)}}\right)^{n-j} = \kappa_n p_{u_K^\pm}(\tan(\alpha))^j \sin(\alpha)^{n-j}.
\end{equation}
Since $p_{u_K^{\pm}}$ must be non-decreasing, we conclude that also
\begin{equation}
\label{eq:p_uK_non-trivdec}
\alpha\mapsto \kappa_n p_{u_K^{\pm}}(\tan(\alpha))^j = \frac{\int_{C_\alpha^\pm} |z_{n+1}| \d S_j(K,z)}{\sin(\alpha)^{n-j}},\quad \alpha\in (0,\tfrac{\pi}{2}],
\end{equation}
must be non-decreasing. Note that by Lemma~\ref{le:subdiff_conjugate} together with \eqref{eq:floor_conjugate}, we retrieve the radius of $\proj_{\Rn} K = \dom(\lfloor K \rfloor)$ from $\lim_{r\to\infty} p_{u_{K}^\pm}(r)$. Thus, when $R_K>0$, the map \eqref{eq:p_uK_non-trivdec} also has to be non-trivial.

\subsection{Further Preparations}
Having established the necessity of the conditions on $F_\mu^\pm$ in Theorem~\ref{thm:main_CM_revol}, we need some further results to prove the rest of the statement. We begin with a trivial estimate.

\begin{lemma}
\label{le:meas_ineq}
Let $j\in\{1,\ldots,n\}$. If $u\in\Conv(\Rn)$ is radially symmetric, then
\[
\MA_j(u;\cdot)\leq 2^{\frac{n-j}{2}} \MA(u[j],u_B[n-j];\cdot)
\]
on $\Rn\setminus D_1^n$.
\end{lemma}
\begin{proof}
Let $u\in\Conv(\Rn)$ be radially symmetric. For every $1\leq s < t$ it follows from Lemma~\ref{le:ma_d_r_n} that
\begin{align*}
\MA_j(u;D_t^n\setminus D_s^n) &= \MA_j(u;D_t^n) - \MA_j(u;D_s^n)\\
&= \frac{\MA(u[j],u_B[n-j];D_t^n)}{\left(\frac{t}{\sqrt{1+t^2}}\right)^{n-j}} - \frac{\MA(u[j],u_B[n-j];D_s^n)}{\left(\frac{s}{\sqrt{1+s^2}}\right)^{n-j}}\\
&\leq 2^{\frac{n-j}{2}} \MA(u[j],u_B[n-j];D_t^n\setminus D_s^n).
\end{align*}
The statement now follows from Corollary~\ref{cor:diks}, applied to the rotationally invariant measures $\chi_{\Rn\setminus D_1^n}(x) \d \MA_j(u;x)$ and $2^{\frac{n-j}{2}}\chi_{\Rn\setminus D_1^n}(x) \d \MA(u[j],u_B[n-j];x)$.
\end{proof}

Next, for $K \in \KN$ and $z \in \sN$ we write
\[
F(K,z)=\{x \in K: \langle x,z \rangle=h_K(z) \}.
\]
for the support set of $K$ with outer unit normal $z$. We remark that
\begin{equation}
\label{eq:support_set_add}
F(K+L,z)=F(K,z)+F(L,z)
\end{equation}
for $K,L\in\KN$ and $z\in\sN$, and, when $K$ is $(n+1)$-dimensional,
\begin{equation}
\label{eq:s_n_support_set}
S_n(K,\omega)=\hm^{n}\left(\bigcup_{z\in\omega} F(K,z)\right)
\end{equation}
for $\omega\subseteq \sN$ Borel. See, for example, \cite[Theorem 1.7.5 (c)]{SchneiderConvexBodiesBrunn2013} and \cite[Theorem 4.2.3]{SchneiderConvexBodiesBrunn2013}, respectively. We use these facts to treat a possible vertical segment contained in the boundary of a body of revolution (cf.\ \cite[Lemma 3.7]{BraunerHofstaetterOrtegaMoreno_mixedCM}).

\begin{lemma}
\label{le:vertical_segments}
If $K\in\KN$ is a body of revolution, then
\begin{equation}
\label{eq:theta_F}
\vartheta F(K,z)= F(K,\vartheta z)
\end{equation}
for every $\vartheta \in \SO(n)$ and $z \in \sN_o$, and, in particular, $F(K,z)$ is a line segment parallel to\linebreak$\ospan\{e_{n+1}\}$. In addition,
\[
\vol_1(F(K,z))\, R_K^{j-1}=\frac{S_j(K,\sN_o)}{j \kappa_n },
\]
for every $z\in\sN_o$ and $j\in\{1,\ldots,n\}$.
\end{lemma}
\begin{proof}
Equation \eqref{eq:theta_F} trivially follows from the definition of $F(K,z)$ together with the assumption that $K$ is a body of revolution. For $z \in \s^n_0$, consider now two (not necessarily distinct) points $x,y \in F(K,z)$. By convexity, we have
\[
[x,y]=\{(1-t)x+ty : t \in [0,1] \} \subseteq F(K,z).
\]
As $\proj_{\Rn} K=B_{R_K}^n$ and since the image of $F(K,z)$ under the same projection is a subset of $\bd(B_{R_K}^n)$, this is only possible if $[x,y]$ lies on a line parallel to $\ospan\{e_{n+1}\}$. This shows that the compact, convex set $F(K,z)$ is of dimension $1$, which means it has to be a line segment. The length of this segment will be denoted by $\ell=\vol_{1}(F(K,z))$, which is independent of $z\in\sN_o$.

Next, observe that
\begin{align*}
\bigcup_{z\in\sN_o} F(K,z) &=\{x\in \bd(K) : K \text{ has an outer unit normal in } \sN_o \text{ at } x\}\\
&= \{x\in \bd(\Gamma_K) : \Gamma_K \text{ has an outer unit normal in } \sN_o \text{ at } x\}\\
&=\bigcup_{z\in\sN_o} F(\Gamma_K,z), 
\end{align*}
where $\Gamma_K$ is a right cylinder of the form 
\[
\Gamma_K = B^{n}_{R_K}+[o,\ell\, e_{n+1}]+c\, e_{n+1}
\]
for some $c\in\R$.
It now follows from \eqref{eq:support_set_add} that 
\[
F(K+ B_\varepsilon^{n+1},z)=F(\Gamma_K + B_{\varepsilon}^{n+1},z)=F(\Gamma_K+(B_{\varepsilon}^{n+1}\cap e_{n+1}^\perp),z)
\]
for every $z \in \sN_o$ and $\varepsilon>0$. Together with \eqref{eq:s_n_support_set} and assuming that $K$ is $(n+1)$-dimensional, this shows
\[
S_n(K+B_{\varepsilon}^{n+1}, \sN_o)=S_n(\Gamma_K+(B_{\varepsilon}^{n+1}\cap e_{n+1}^\perp), \sN_o),
\]
and thus, by polynomial expansion,
\[
S_j(K, \sN_o)=\bar{S}_j(\Gamma_K,\sN_o)
\]
for every $j \in \{0,1,\dots,n \}$, where we recall the definition of $\bar{S}_j$ in \eqref{eq:bar_s_j}. The proof is now concluded by Theorem~\ref{thm:bar_sk_rewrite}, which shows that
\[
\bar{S}_j(\Gamma_K,\sN_o)=\frac{\kappa_n}{\kappa_j} S_j(B_{R_K}^j+[o,\ell\, e_{n+1}],\s^j_o) = \frac{\kappa_n}{\kappa_j} j \kappa_j R_K^{j-1} \ell,
\]
where we have taken into account that $\Gamma_K$ is invariant under $\SO(n)$. Lastly, in case $K$ is of lower dimension, it follows from rotational symmetry that either $K$ is a disk in $e_{n+1}^\perp$ or a line segment in $\ospan\{e_{n+1}\}$. In case of the former, we have $\vol_1(F(K,z))=0$ for $z\in\sN_o$ but also $S_j(K,\sN_o)=0$, which is straightforward to check. In case of the line segment, we have $R_K=0$ and, similarly, $S_j(K,\sN_o)=0$.
\end{proof}

\subsection{Proof of Theorem~\ref{thm:main_CM_revol}}
In case there exists a body of revolution $K\in\KN$ with $R_K>0$ and such that $S_j(K,\cdot)=\mu$, it follows from \eqref{eq:p_uK_non-trivdec} that $F_{\mu}^\pm$ must be non-trivial and non-decreasing functions on $(0,\frac{\pi}{2}]$.

Conversely, let $\mu$ be such that the associated functions $F_{\mu}^\pm$ are non-trivial and non-decreasing. By the Riesz--Markov--Kakutani representation theorem, there exist finite Borel measures $\eta^+$ and $\eta^-$ on $\Rn$ such that
\[
\int_{\Rn}\varphi(x) \d\eta^\pm(x) = \int_{\sN_{\pm}} \varphi(\gnom(z)) |z_{n+1}| \d\mu(z)
\]
for every $\varphi\in C_c(\R^n)$. Since $\mu$ is $\SO(n)$ invariant, the measures $\eta^\pm$ are rotationally invariant. Furthermore, by \eqref{eq:chi_d_tan},
\[
\eta^\pm(D_{\tan(\alpha)}^n)=\int_{C_\alpha^\pm} |z_{n+1}| \d\mu(z) = \sin(\alpha)^{n-j} F_{\mu}^{\pm}(\alpha)
\]
for every $\alpha\in (0,\frac{\pi}{2}]$. Since $\sin(\arctan(r))=r/\sqrt{1+r^2}=p_{u_B}(r)$ for $r>0$, the functions
\[
r\mapsto  \frac{\eta^\pm(D_r^n)}{p_{u_B}(r)^{n-j}} = F_{\mu}^{\pm}(\arctan(r)),\quad r>0,
\]
are non-increasing, and thus, by Theorem~\ref{thm:mixed_hessian_unbounded}, there exist unique radially symmetric convex functions $u^\pm\in\Conv(\Rn)$, solving
\[
\begin{cases}
\MA(u^\pm[j],u_{B}[n-j];\cdot)=\eta^\pm &\text{ in } \Rn, \\
u^\pm(o)=0,
\end{cases}
\]
given through $u^\pm(x)=\int_0^{|x|} \left(\frac{1}{\kappa_n} F_{\mu}^\pm(\arctan(t)) \right)^{\frac 1j} \d t$. Since the functions $F_\mu^\pm$ are bounded, the subdifferentials of $u^\pm$ are also bounded and, in particular, the sets $(\partial u^\pm)(\Rn)$ are open or closed balls of radius
\begin{equation}
\label{eq:R_mu}
R_\mu = \left(\tfrac{1}{\kappa_n} F_\mu^+(\tfrac{\pi}{2})\right)^{\frac 1j}=\left(\tfrac{1}{\kappa_n} F_\mu^-(\tfrac{\pi}{2})\right)^{\frac 1j}>0,
\end{equation}
where the second equality is a consequence of the assumption that $\mu$ is centered.

Next, we consider the convex conjugates $(u^\pm)^*$. By Lemma~\ref{le:subdiff_conjugate} and \eqref{eq:R_mu}, the domains of these functions are open or closed balls of radius $R_\mu$. We will now show that the $n$-dimensional areas of their graphs (when restricting the functions to their domains) are bounded. Since convex conjugates are always lower-semicontinuous, this will then imply that the domains $\dom((u^\pm)^*)$ are closed and that the functions $(u^\pm)^*$ are bounded on their respective domains. Since $\mu$ is a finite measure, it follows from Lemma~\ref{le:meas_ineq} that
\begin{align*}
    \infty >\mu(\sN_{\pm})&=\int_{\sN_{\pm}}|z_{n+1}|\frac{1}{|z_{n+1}|} \d \mu(z)\\
    &=\int_{\Rn}\sqrt{1+|x|^2} \d\eta^\pm(x)\\
    &=\int_{\Rn} \sqrt{1+|x|^2} \d \MA(u^\pm[j],u_B[n-j];x)\\
    &\geq \int_{\Rn\setminus D_1^n} \sqrt{1+|x|^2} \d \MA(u^\pm[j],u_B[n-j];x)\\
    &\geq \frac{1}{2^{\frac{n-j}{2}}} \int_{\Rn\setminus D_1^n} \sqrt{1+|x|^2} \d\MA_j(u^\pm;x).
\end{align*}
Since $\int_{D_1^n} \sqrt{1+|x|^2} \d\MA_j(u^\pm;x)$ is trivially finite, it now follows from Theorem~\ref{thm:ck_ma_SOn}, the radial symmetry of $u^\pm$, \eqref{eq:conjugate_MA}, and Lemma~\ref{le:conjugate_projection} that
\begin{align*}
\infty > \int_{\Rn} \sqrt{1+|x|^2}\d\MA_j(u^\pm;x) &= \frac{\kappa_n}{\kappa_j} \int_{\R^j}\sqrt{1+|x|^2}\d\MA'(u^\pm\vert_{\R^j};x)\\
&=\frac{\kappa_n}{\kappa_j} \int_{\dom(\proj_{\R^j} (u^\pm)^*)}\sqrt{1+|\nabla \proj_{\R^j}(u^\pm)^*(x)|^2}\d x.
\end{align*}
This shows that the functions $\proj_{\R^j} (u^\pm)^*$ have graphs with finite $j$-dimensional area, which by radial symmetry implies that the functions $(u^\pm)^*$ have finite $n$-dimensional area. In particular, $(u^\pm)^*$ are radially symmetric convex functions with compact domain $B_{R_{\mu}}^n$.

We have thus found two convex functions $u^\pm$ whose conjugates are, up to the addition of constants, candidates for $\lfloor K\rfloor$ and $\lceil K \rceil$. Recall that our functions were obtained from Theorem~\ref{thm:mixed_hessian_unbounded}, which gives solutions that vanish at the origin. Since area measures are invariant under translations, we will therefore fix $u^-$ and add a suitable constant to $u^+$. First, we note that $u^\pm(o)=0$ together with \eqref{eq:def_conjugate} and radial symmetry implies that also $(u^\pm)^*(o)=0$. This means that the maximums of $(u^\pm)^*$ on their domains, which are obtained at any point with norm $R_\mu$, also tell us how high (in terms of the length of the projection onto $\ospan\{e_{n+1}\}$) the parts of the boundary of $K$ with outer unit normals in $\sN_{\pm}$ are. By \eqref{eq:def_conjugate} we have
\begin{align*}
(u^\pm)^*(R_\mu e_1)&=\sup\nolimits_{r\in \R} \left(r R_\mu-u^\pm(r e_1) \right)\\
&=\sup\nolimits_{r\in [0,\infty)} \left(r \left(\tfrac{1}{\kappa_n} F_\mu^\pm\left(\tfrac{\pi}{2}\right) \right)^{\frac 1j}-\int_0^r \left(\tfrac{1}{\kappa_n} F_\mu^\pm(\arctan(t)) \right)^{\frac 1j} \d t \right)\\
&=\sup\nolimits_{\beta\in[0,\frac{\pi}{2})} \left(\tan(\beta) \left(\tfrac{1}{\kappa_n} F_\mu^\pm\left(\tfrac{\pi}{2}\right) \right)^{\frac 1j} - \int_0^\beta \frac{1}{\cos(s)^2}\left(\tfrac{1}{\kappa_n} F_\mu^\pm(s) \right)^{\frac 1j} \d s\right),
\end{align*}
where we have taken the radial symmetry of $u^\pm$ into account. Finally, we have to consider those parts whose outer normal vectors lie in $\sN_o$. According to Lemma~\ref{le:vertical_segments}, the vertical segments in the boundary of any solution $K$ need to be of length
\[
\ell_K = \frac{S_j(K,\sN_o)}{j \kappa_n R_K^{j-1}}=\frac{\mu(\sn_o)}{j\kappa_n R_\mu^{j-1}}.
\]
Considering \eqref{eq:R_mu} and \eqref{eq:retrieve_K_floor_ceil}, we therefore set
\begin{equation}
\label{eq:solution_K}
K=\{(x,t)\in\Rn\times \R : u^-(x)\leq t \leq -u^+(x)+ c_\mu\},
\end{equation}
where $c_\mu$ is as in the statement. By our construction together with Lemma~\ref{le:int_uk} and Lemma~\ref{le:vertical_segments}, we obtain $S_j(K,\cdot)=\mu$.

Now assume that $\tilde{K}\in\KN$ is such that $S_j(\tilde{K},\cdot)=\mu$. By \eqref{eq:p_uK_non-trivdec} we must have
\[
p_{u_{\tilde{K}}^\pm}=p_{u_K^\pm},
\]
which by \eqref{eq:p_u_left_derviative} and Lemma~\ref{thm:rockafellar_integegrate} shows that $u_{\tilde{K}}^\pm$ and $u_K^\pm$ must coincide up to the addition of constants. Consequently, by \eqref{eq:floor_conjugate} there exist constants $c^+,c^-\in\R$ such that
\[
\lfloor K \rfloor = \lfloor \tilde{K}\rfloor + c^-\quad \text{and}\quad \lceil K \rceil = \lceil \tilde{K}\rfloor + c^+.
\]
By \eqref{eq:retrieve_K_floor_ceil} together with Lemma~\ref{le:vertical_segments} this implies that $K$ and $\tilde{K}$ coincide up to a translation along $\ospan\{e_{n+1}\}$.

Finally, let us give an explicit description of the body $K$ as in \eqref{eq:solution_K}, which is the unique solution such that $h_K(-e_{n+1})=0$. For $z\in\sN_-$ we have
\[
h_K(z)=|z_{n+1}|\,h_K(\gnom(z),-1)=|z_{n+1}|\,u^-(\gnom(z))=|z_{n+1}|\int_0^{|\gnom(z)|}\left(\tfrac{1}{\kappa_n} F_\mu^-(\arctan(t)) \right)^{\frac 1j} \d t.
\]
Choosing $\alpha_z\in[0,\tfrac{\pi}{2})$ such that $|z_{n+1}|=\cos(\alpha_z)$, this becomes
\begin{equation}
\label{eq:h_K_cos_alpha}
h_K(z)=\cos(\alpha_z) \int_0^{\tan(\alpha_z)}\left(\tfrac{1}{\kappa_n} F_\mu^-(\arctan(t)) \right)^{\frac 1j} \d t = \cos(\alpha_z) \int_0^{\alpha_z}\frac{1}{\cos(s)^2}\left(\tfrac{1}{\kappa_n} F_\mu^-(s) \right)^{\frac 1j} \d s.
\end{equation}
On $\sN_+$ we obtain an analogous expression, but we must now account for the height of $K$, meaning we need to have $h_K(e_{n+1})=c_\mu$. This results in an additional term of the form $\langle c_\mu \, e_{n+1}, z\rangle = \cos(\alpha_z) c_\mu$. Lastly, on $\sN_o$ we must have, by construction,
\[
h_K(z)=R_\mu=\left(\tfrac{1}{\kappa_n} F_\mu^-\left(\tfrac{\pi}{2}\right) \right)^{\frac 1j},
\]
which matches with the limit $\alpha_z\to \frac{\pi}{2}$ in \eqref{eq:h_K_cos_alpha}.
\qed

\section{Further Discussion}
\label{se:discussion}
\subsection{Examples}
Let us briefly demonstrate Theorem~\ref{thm:main_CM_revol} on two examples. First, let $\mu$ be the $j$th area measure of the $(n+1)$-dimensional Euclidean unit ball. In this case, $\mu$ is proportional to the $n$-dimensional Hausdorff measure restricted to $\sN$, that is,
\[
\d\mu(z)=\d S_j(B_1^{n+1},z)=\frac{1}{(n+1)\kappa_{n+1}} \d \hm^n(z)
\]
on $\sN$. By \eqref{eq:int_moment_cap} we have
\[
\int_{C_\alpha^\pm} |z_{n+1}|\d\mu(z) = \kappa_n \sin(\alpha)^n
\]
and thus
\[
F_\mu^{\pm}(\alpha) = \kappa_n \sin(\alpha)^j
\]
for $\alpha\in(0,\tfrac{\pi}{2}]$. The formula for $h_K$ in the statement of Theorem~\ref{thm:main_CM_revol} now gives
\[
h_K(z)=\cos(\alpha_z) \int_0^{\alpha_z} \frac{\sin(s)}{\cos(s)^2} \d s = \cos(\alpha_z)\left(\frac{1}{\cos(\alpha_z)}-1 \right)= 1+\langle e_{n+1},z\rangle
\]
for $z\in\sN_-$, and an analogous expression is obtained on the other parts of the sphere. This shows that $K=B_1^{n+1}+e_{n+1}$.

\medskip

As a second example, we choose $\mu$ to be the $j$th area measure of the $n$-dimensional disk $B_1^{n+1}\cap e_{n+1}^\perp$. When $j=n$, then $\mu$ is concentrated on $\pm e_{n+1}$, each with measure $\kappa_n$. This trivially implies that
\[
F_\mu^\pm\equiv \kappa_n
\]
is constant. Similarly, when $j\in\{1,\ldots,n-1\}$, equation \eqref{eq:int_moment_cap} shows that
\[
\int_{C_\alpha^\pm}|z_{n+1}| \d\mu(z) = \kappa_n \sin(\alpha)^{n-j}
\]
for $\alpha\in(0,\tfrac{\pi}{2}]$, resulting again in $F_{\mu}^\pm\equiv \kappa_n$. For the support function of the solution $K$ we now obtain
\begin{align*}
h_K(z)&=\cos(\alpha_z) \int_0^{\alpha_z} \frac{1}{\cos(s)^2} \d s = \cos(\alpha_z) \tan(\alpha_z) = \sin(\alpha_z)\\
&= \sqrt{1-\cos(\alpha_z)^2}= \sqrt{1-|z_{n+1}|^2}= \sqrt{z_1^2+\cdots + z_n^2}
\end{align*}
for $z\in\sN$, and thus, $K=B_1^{n+1}\cap e_{n+1}^\perp$. 

\subsection{On Mixed Problems}
Let us emphasize again that Theorem~\ref{thm:main_CM_revol} is based on the solution to entire Monge--Amp\`ere-type equations from Theorem~\ref{thm:mixed_hessian_unbounded}. Since the latter is formulated for more general mixed measures, it is, in principle, also straightforward to generalize Theorem~\ref{thm:main_CM_revol} to study existence and uniqueness problems for mixed area measures, as presented in \cite[Section 6.3]{BraunerHofstaetterOrtegaMoreno_mixedCM}. More precisely, given reference bodies of revolution $K_{j+1},\ldots,K_n\in\KN$ we are looking for necessary and sufficient conditions on a finite, $\SO(n)$ invariant Borel measure $\mu$ on $\sN$ such that
\[
S(K[j],K_{j+1},\ldots,K_n,\cdot)=\mu
\]
for some body of revolution $K\in\KN$. Considering the results of Section~\ref{se:mm_sc}, to do so it is essentially enough to redefine the functions $F_\mu^\pm$ as
\begin{equation}
\label{eq:f_mu_mixed}
F_\mu^\pm(\alpha)=\frac{\int_{C_\alpha^\pm} |z_{n+1}| \d\mu(z)}{p_{u_{K_{j+1}}^\pm}(\tan(\alpha))\cdots p_{u_{K_n}^\pm}(\tan(\alpha))},\quad \alpha\in (0,\tfrac{\pi}{2}].
\end{equation}
However, analogous to the discussion in Section~\ref{se:explicit_solutions}, one now needs to carefully treat the cases when one of the functions $p_{u_{K_i}^\pm}$ with $i\in\{j+1,\ldots,n\}$ vanishes on some interval of the form $(0,r_o)$, which means that $K_i$ has a cusp in direction $\pm e_{n+1}$ and in which case one generally loses uniqueness of solutions (even up to translations; cf.\ \cite[Remark 6.14]{BraunerHofstaetterOrtegaMoreno_mixedCM}). Furthermore, a corresponding analogue to Lemma~\ref{le:vertical_segments} needs to be adapted if $F(K_i,z)>0$ for some $i\in\{j+1,\ldots,n\}$ and $z\in\sN_o$. We plan to address such mixed problems more systematically in future work.

\medskip

That being said, let us highlight the special case $K_{j+1}=\cdots=K_n=B_{1}^{n+1}\cap e_{n+1}^\perp$, which corresponds to the existence and uniqueness problem for the measures $\bar{S}_j(K,\cdot)$. In this case, $p_{u_{K_i}^\pm}\equiv 1$ for every $i\in\{j+1,\ldots,n\}$ and \eqref{eq:f_mu_mixed} becomes
\[
F_\mu^\pm(\alpha)=\int_{C_{\alpha}^\pm} |z_{n+1}|\d\mu(z).
\]
Note that this function is always non-increasing, and it is trivial if and only if the $\SO(n)$ invariant measure $\mu$ is concentrated on $\sN_o$. Since $\bar{S}_j(K,\sN_o)=S_j(K,\sN_o)$ (cf.\ the proof of Lemma~\ref{le:vertical_segments}), we therefore obtain the following result, where we remark that the same necessary and sufficient conditions on the measure $\mu$ were recently obtained in \cite[Theorem B]{BraunerHofstaetterOrtegaMoreno_mixedCM}, using an independent approach (see also Section~\ref{se:alt_proof}).

\begin{theorem}
	\label{thm:main_bar_Sj_revol}
	Let $j\in\{1,\ldots,n\}$ and let $\mu$ be a finite, centered, $\SO(n)$ invariant Borel measure on $\sN$. There exists a convex body of revolution $K\in\KN$ with $R_K>0$ such that
	\[
	\bar{S}_j(K,\cdot)=\mu,
	\]
	if and only if $\mu$ is not concentrated on $\sN_o$. The body $K$ is unique up to translations along $\ospan\{e_{n+1}\}$ and its support function, apart from the addition of $\langle \tau\, e_{n+1},z\rangle$ with some $\tau\in\R$, is given by
	\[
	h_K(z)=\begin{cases} \cos(\alpha_z) \int_0^{\alpha_z}\frac{1}{\cos(s)^2}\left(\tfrac{1}{\kappa_n} \int_{C_s^-} |\nu_{n+1}|\d\mu(\nu) \right)^{\frac 1j} \d s,\qquad &\text{if } z\in \sN_-,\\
		\left(\tfrac{1}{\kappa_n} \int_{\sN_-} |\nu_{n+1}|\d\mu(\nu) \right)^{\frac 1j},\qquad &\text{if } z\in \sN_o,\\
		\cos(\alpha_z)\left( \int_0^{\alpha_z}\frac{1}{\cos(s)^2}\left(\tfrac{1}{\kappa_n} \int_{C_s^+} |\nu_{n+1}|\d\mu(\nu) \right)^{\frac 1j} \d s+c_\mu \right),\qquad &\text{if } z\in \sN_+,
	\end{cases}
	\]
	where $\alpha_z\in [0,\tfrac{\pi}{2}]$ is such that $|z_{n+1}|=\cos(\alpha_z)$, and where
	\begin{align*}
		c_\mu &= \frac{\mu(\sN_o)}{j\kappa_n \left(\tfrac{1}{\kappa_n} \int_{\sN_-} |\nu_{n+1}|\d\mu(\nu)\right)^{\frac{j-1}{j}}} \\
		&\quad+\sup\nolimits_{\beta\in[0,\frac{\pi}{2})} \left(\tan(\beta) \left(\tfrac{1}{\kappa_n} \int_{\sN_-} |\nu_{n+1}|\d\mu(\nu) \right)^{\frac 1j} - \int_0^\beta \frac{1}{\cos(s)^2}\left(\tfrac{1}{\kappa_n} \int_{C_s^-} |\nu_{n+1}|\d\mu(\nu) \right)^{\frac 1j} \d s\right)\\
		&\quad + \sup\nolimits_{\gamma\in[0,\frac{\pi}{2})} \left(\tan(\gamma) \left(\tfrac{1}{\kappa_n} \int_{\sN_+} |\nu_{n+1}|\d\mu(\nu) \right)^{\frac 1j} - \int_0^\gamma \frac{1}{\cos(s)^2}\left(\tfrac{1}{\kappa_n} \int_{C_s^+} |\nu_{n+1}|\d\mu(\nu) \right)^{\frac 1j} \d s\right).
	\end{align*}
\end{theorem}

Lastly, following the spirit of the classical solution to the Christoffel problem, the existence and uniqueness question for the measure $\bar{S}_1(K,\cdot)$, without additional symmetry assumptions, was recently settled in \cite{BraunerHofstaetterOrtegaMoreno_diskC}.

\subsection{Disintegration}
\label{se:disintegration}
Let $\mu$ be a finite, $\SO(n)$ invariant Borel measure on $\sN$. Expanding on the methods of Lemma~\ref{le:rot_inv_meas_determined}, it is elementary to check that there exists a unique finite Borel measure $\tilde{\mu}$ on $\s^1 \subset \ospan\{e_1,e_{n+1}\}$ that is invariant under $(x_1,x_{n+1})\mapsto (-x_1,x_{n+1})$ and such that
\[
\int_{\sN} \varphi(z) \d\mu(z) = \int_{\SO(n)} \int_{\s^1} \varphi(\vartheta y) \d\tilde{\mu}(y) \d\vartheta
\]
for every measurable $\varphi\colon \sN\to[0,\infty)$. Given the properties of $\tilde{\mu}$ and the structure of $\s^1$, this is equivalent to considering a measure $\bar{\mu}$ on $[-\tfrac{\pi}{2},\tfrac{\pi}{2}]$ such that $\mu(C_\alpha^-)=\bar{\mu}([-\tfrac{\pi}{2},\alpha-\tfrac{\pi}{2}))$ for $\alpha\in(0,\tfrac{\pi}{2}]$ and a similarly relation holds for $C_\alpha^+$ (cf.\ \cite[Proposition 6.1]{BraunerHofstaetterOrtegaMoreno_mixedCM}). The functions $F_\mu^\pm$ in Theorem~\ref{thm:main_CM_revol} can now be rewritten as
\[
F_\mu^-(\alpha)=\frac{-\int_{[-\pi/2,\alpha-\pi/2)} \sin(s) \d\bar{\mu}(s) }{\sin(\alpha)^{n-j}} \quad \text{and} \quad F_\mu^+(\alpha)=\frac{\int_{(\pi/2-\alpha,\pi/2]} \sin(s) \d\bar{\mu}(s) }{\sin(\alpha)^{n-j}},\quad \alpha\in (0,\tfrac{\pi}{2}].
\]

\subsection{An Alternative Proof Strategy}
\label{se:alt_proof}
As mentioned in the introduction, a complete solution to the Christoffel--Minkowski problem for bodies of revolution was recently obtained by Brauner, Hofst\"atter, and Ortega-Moreno in \cite{BraunerHofstaetterOrtegaMoreno_mixedCM}. The conditions on $\mu$ established in Theorem~\ref{thm:main_CM_revol} are essentially an integrated form of those in \cite[Corollary A]{BraunerHofstaetterOrtegaMoreno_mixedCM}. Let us briefly comment on how the overall strategy presented in \cite{BraunerHofstaetterOrtegaMoreno_mixedCM} can be adopted to recover the conditions on $\mu$ as in Theorem~\ref{thm:main_CM_revol}, while entirely bypassing the mixed Monge--Amp\`ere measures that play a central role in our proof presented in Section~\ref{se:CM_problems}.

\medskip

First, we proceed as in \cite[Section 6.1]{BraunerHofstaetterOrtegaMoreno_mixedCM} to solve the existence and uniqueness problem for the measure $\bar{S}_j(K,\cdot)$, using disintegration (cf.\ Section~\ref{se:disintegration}) together with the Kubota-type formula Theorem~\ref{thm:bar_sk_rewrite} and the solution to the classical Minkowski problem, thereby establishing a non-explicit version of Theorem~\ref{thm:main_bar_Sj_revol}. Next, we use \cite[Corollary 3.11]{BraunerHofstaetterOrtegaMoreno_ZonalVal} to obtain the relevant connection between the measures $S_j(K,\cdot)$ and $\bar{S}_j(K,\cdot)$, namely
\begin{equation}
\label{eq:relation_measures}
\frac{\int_{C_\alpha^{\pm}} |z_{n+1}|\d S_j(K,z)}{\sin(\alpha)^{n-j}}=\int_{C_\alpha^{\pm}} |z_{n+1}|\d \bar{S}_j(K,z)
\end{equation}
for $\alpha\in (0,\tfrac{\pi}{2}]$, which is also contained in \cite[Proposition 5.8]{BraunerHofstaetterOrtegaMoreno_mixedCM}. See also \cite[Section 5]{HugMussnigUlivelli2} for an analytic analogue as well as Lemma~\ref{le:ma_k_phi_k}.

The key observation is now that every $\SO(n)$ invariant Borel measure on $\sN_{\pm}$ is uniquely determined by two non-increasing functions on $(0,\tfrac{\pi}{2}]$, which follows from using Lebesgue--Stieltjes measures and disintegration again. For the measure $|z_{n+1}|\d \bar{S}_j(K,z)$, these functions are
\[
\alpha\mapsto \int_{C_\alpha^{\pm}} |z_{n+1}|\d \bar{S}_j(K,z),\quad \alpha\in (0,\tfrac{\pi}{2}].
\]
Together with \eqref{eq:relation_measures}, the equality $S_j(K,\cdot)\equiv \bar{S}_j(K,\cdot)$ on $\sN_o$, and a finiteness estimate (similar to Lemma~\ref{le:meas_ineq}), this gives precisely the conditions on $\mu$ presented in Theorem~\ref{thm:main_CM_revol}.

\subsection{Further Observations}
Similar to Firey's solution of the Christoffel--Minkowski problem for sufficiently regular bodies of revolution \cite{Firey_revolution}, Theorem~\ref{thm:main_CM_revol} easily shows that if $\mu$ is the $j$th area measure of some body of revolution $K\in\KN$ with $R_K>0$, then $\mu$ is also the $k$th area measure of some $L\in\KN$ with $R_L>0$ and $k\in\{j+1,\ldots,n\}$.

In addition, since the functions $F_\mu^\pm$ in Theorem~\ref{thm:main_CM_revol} depend linearly on $\mu$, it follows that if $K_1,K_2\in\KN$ are bodies of revolution with $R_{K_1},R_{K_2}>0$, then there exists a body of revolution $L\in\KN$ with $R_L>0$ such that
\[
S_j(L,\cdot)=S_j(K_1,\cdot)+S_j(K_2,\cdot).
\]
For $j=1$, the body $L$ is (up to translations along $\ospan\{e_{n+1}\}$) just the Minkowski sum of $K_1$ and $K_2$, and for $j=n$, we obtain $L$ as the Blaschke sum. In general, however, the sum of two $j$th area measures is not the $j$th area measure of a convex body \cite[Theorem 8.4.1]{SchneiderConvexBodiesBrunn2013}.

\subsection*{Acknowledgments}
Fabian Mussnig was supported by the Austrian Science Fund (FWF):\ 10.55776/P36210. Jacopo Ulivelli was supported by the Austrian Science Fund (FWF):\ 10.55776/P34446.


\vfill

\parbox[t]{8.5cm}{
Fabian Mussnig\\
Institut f\"ur Diskrete Mathematik und Geometrie\\
TU Wien\\
Wiedner Hauptstra{\ss}e 8-10/1046\\
1040 Wien, Austria\\
e-mail: fabian.mussnig@tuwien.ac.at

\smallskip
\textit{Current address:} Mathematics Department\\
University of Salzburg\\
Hellbrunner Stra{\ss}e 34\\
5020 Salzburg, Austria\\
e-mail: fabian.mussnig@plus.ac.at
}

\bigskip

\parbox[t]{8.5cm}{
Jacopo Ulivelli\\
Institut f\"ur Diskrete Mathematik und Geometrie\\
TU Wien\\
Wiedner Hauptstra{\ss}e 8-10/1046\\
1040 Wien, Austria\\
e-mail: jacopo.ulivelli@tuwien.ac.at}

\end{document}